\documentclass[12pt]{amsart}

\usepackage{amssymb,url,mathrsfs,euscript}

\newtheorem{teo}{Theorem}[section]
\newtheorem{lema}{Lemma}[section]
\newtheorem{pro}{Proposition}[section]

\newtheorem{rmk}{Remark}[section]\newtheorem*{rmk*}{Remark}

\theoremstyle{remark}     
\newtheorem{ex}[rmk]{Example}\newtheorem*{ex*}{Example}
\newtheorem*{exs*}{Examples}

\def\sideremark#1{\ifvmode\leavevmode\fi\vadjust{\vbox to0pt{\vss
\hbox to 0pt{\hskip\hsize\hskip1em%
\vbox{\hsize2cm\tiny\raggedright\pretolerance10000%
\noindent #1\hfill}\hss}\vbox to8pt{\vfil}\vss}}}%

\swapnumbers              
\theoremstyle{plain}      
\newtheorem{lemma}[rmk]{Lemma}

\newtheorem{theorem}[rmk]{Theorem}

\theoremstyle{definition} 

\newcommand{\bt}{\begin{theorem}}\newcommand{\et}{\end{theorem}}
\newcommand{\bl}{\begin{lemma}}\newcommand{\el}{\end{lemma}}
\newcommand{\bp}{\begin{proof}}\newcommand{\ep}{\end{proof}}

\newcommand{\be}{\begin{equation}}\newcommand{\ee}{\end{equation}}
\newcommand{\bdm}{\begin{displaymath}}
\newcommand{\edm}{\end{displaymath}}

\numberwithin{equation}{section}

\def \bnabla{\overline{\nabla}}

\def \z{\zeta}

\newcommand{\lto}{\longrightarrow}

\newcommand{\D}{\curly{D}}
\renewcommand{\O}{\Omega}

\renewcommand{\o}{\omega}

\renewcommand{\geq}{\geqslant}
\renewcommand{\leq}{\leqslant}
\newcommand{\w}{\wedge}
\renewcommand{\ln}{\textsl{ln}\,}

\DeclareMathOperator{\Ker}{Ker}

\renewcommand{\Re}{\textsl{Re\,}}
\renewcommand{\Im}{\textsl{Im\,}}

\renewcommand{\span}{\textsl{span}}

\newcommand{\curly}{\mathscr}

\newcommand{\q}{\quad}

\newcommand{\bb}{\mathbb}
\newcommand{\ba}{\begin{array}}\newcommand{\ea}{\end{array}}

\renewcommand{\&}{{\footnotesize \&}}

\begin{document}
\title[]{Complex homothetic foliations on K\"ahler manifolds}
\thanks{Partially supported by the {\sc Prin} \oldstylenums{2007} ``Geometria differenziale ed analisi globale'' of {\sc Miur}, the Royal Society of New Zealand, Marsden grant no. 06-UOA-029, and {\sc Gnsaga} of INdAM.}
\author[S. G. Chiossi]{Simon G. Chiossi} 
\address[SGC]{Dipartimento di matematica, Politecnico di Torino, c.so Duca degli Abruzzi 24, 10129 Torino, Italy}
\email{simon.chiossi@polito.it}

\author[P.- A. Nagy]{Paul-Andi Nagy} 
\address[PAN]{Institut f\"ur Mathematik und Informatik, Ernst-Moritz-Arndt Universit\"at Greifs\-wald, Walther-Rathenau str.47, 
D-17487 Greifswald}
\email{nagyp@uni-greifswald.de}

\frenchspacing

\begin{abstract}
We obtain a local classification of complex homothetic foliations  on K\"ahler manifolds by complex curves. This is used to construct almost Kaehler, 
Ricci-flat metrics subject to additional curvature properties.
\end{abstract}

\maketitle

\tableofcontents
\section{Introduction}
Consider a foliation $\mathcal{F}$ on a Riemannian manifold $(M,g)$. It is called \emph{conformal} if 
$$ 
L_Vg=\theta(V)g
$$
holds on $T\mathcal{F}^{\perp}$ for any vector $V$ tangent to the leaves where $\theta$ is a $1$-form on $M$ that vanishes on $T\mathcal{F}^{\perp}$. The foliation $\mathcal{F}$ is called {\it{homothetic}} if $\theta$ is closed, that is $d\theta=0$. 
These natural classes of foliations, as introduced in \cite{V2}, provide a geometric set-up for integrating the Einstein \cite{wood} or Einstein-Weyl \cite{Pedersen-Swann} equations. Moreover, homothetic foliations 
are important sources of harmonic morphisms, see \cite{Baird-Wood-book, pant}. Examples include Riemannian metrics of 
warped-product type \cite{sven}, and the case when the leaves are of dimension, or co-dimension, one is fairly well understood \cite{pant,Bryant}. 

When $(M^{2m},g,J)$ is Kaehler and $\mathcal{F}$ is a complex Riemannian foliation, or equivalently $\theta$ vanishes, the situation has been studied in \cite{nagy1,nagy2}. However, homothetic 
foliations with complex leaves on K\"ahler manifolds are still mysterious and their classification an open problem. If the manifold is foliated by complex curves we prove:
\begin{teo} \label{hmt} 
A connected K\"ahler manifold equipped with a complex homothetic foliation by curves is either 
\begin{itemize}
\item[(i)] locally obtained from a K\"ahler manifold of Calabi type via a transversely holomorphic twist;
\item[or]
\item[(ii)] the local Riemannian product of a Riemann surface and a K\"ahler manifold;
\item[or]
\item[(iii)] the local Riemannian product of a K\"ahler manifold and a twistor space over a quaternion K\"ahler manifold with positive scalar curvature.
\end{itemize}
\end{teo}
\noindent 
The name Calabi type refers essentially to Calabi's construction of Kaehler metrics on 
a holomorphic line bundle over a K\"ahler base. Simple examples thereof include metric cones of Sasakian manifolds, and an abstract characterisation for surfaces was obtained in \cite{WSD}. 
A twist is simply a smooth map from the K\"ahler manifold to the open unit disc in the complex plane. In (i) above locally means around each point of some dense open subset in $M$.

By using results in \cite{nagy1,nagy2} we first show that that either $\mathcal{F}$ is holomorphic or Riemannian and totally geodesic and that the last instance corresponds to cases (ii) and (iii) in theorem \ref{hmt}. 
The remaining case to treat is when $\mathcal{F}$ is holomorphic, however it is the symplectic perspective 
that suits the problem best. Indeed, the aforementioned twist is actually a deformation of the K\"ahler structure along the leaves of the foliation only, that keeps the symplectic form fixed. If one assumes that the leaves 
are of complex dimension one, the well-known parametrisation of twists linearises in a nice way under certain integrability conditions. In fact, the remaining of the proof of theorem \ref{hmt} is based on the inspection 
of the structure equations of holomorphic  homothetic foliations by curves, which can be solved explicitly. This allows to find a twist that renders the Lee form $\theta$ of $\mathcal{F}$ dual to the gradient of the moment map of a 
certain holomorphic Killing vector field. The familiar reduction procedure \cite{LeBrun:nCP2, P-P} for circle-symmetric K\"ahler metrics then completes the argument. 

The twisted Calabi construction's symplectic nature  suggests how one should obtain almost K\"ahler metrics with curvature properties close to K\"ahler manifolds. In the final section of 
the paper we use the recipe of theorem \ref{hmt} to present new instances of almost K\"ahler structures of so-called class $\mathcal{AK}_3$ in the sense of \cite{Gray, Tri2}. This class has been vigourously studied in dimension $4$ by Apostolov 
et.al \cite{AAD, Apostolov-AD:integrability}; in higher dimension locally irreducible examples are supported by twistor spaces over negative quaternion-K\"ahler manifolds \cite{al}. If one looks at 
dimension $4$ these metrics are Einstein (actually Ricci-flat) precisely when they arise, in the form of the Gibbons-Hawking Ansatz, from a translation-invariant harmonic 
function; see \cite{Apostolov-AD:integrability, Nurowski-P:aK-nK} for details. We show that iterating Calabi's construction with a twist yields examples of another kin, 
and in any dimension, of Einstein(Ricci flat) almost-K\"ahler, non-K\"ahler metrics in the class $\mathcal{AK}_3$. Non Ricci-flat, homogeneous, examples of almost-K\"ahler, non-K\"ahler metrics in 
dimension $4m, m\geq 6$ can be found in \cite{Al}(see also \cite{ADM}).
\section{Complex homothetic foliations} 
\subsection{K\"ahler structures of Calabi type} \label{c-s} 
Following  \cite{WSD, Hwang-S:} we briefly outline the main features of the construction borne in on the seminal article \cite{cal}.

Let $(N,g_{N},I_{N})$ be a K\"ahler manifold with fundamental $2$-form $\o_{N}=g_N(I_N\cdot, \cdot)$. In this setting we will use the convention 
$ I_N\alpha=\alpha(I_N \cdot)$ for the action of $I_N$ on one forms. 
Now take  a Hermitian line bundle $(L,h) \lto N$ and a Hermitian connection $D^L$ on it with 
curvature tensor $-i\o_{N} \otimes 1_L$, and let $M=L^{\times}$ be the total space of $L$ without the zero section. Choosing $D^L$ induces a splitting 
$$TM=\D^{+} \oplus \D^{-}$$ 
where $\D^{+}$ is the distribution tangent to the fibres of $L$ and $\D^{-}$ its orthogonal complement. 
If we let $r:M \to \mathbb{R}$ denote the fibres' norm, and call $2K$ the generator of the natural circle action on $M$, so that $\exp(2\pi i K)=-1$, the fibres' inner product induces a Hermitian metric $h$ on $\D^{+}$ such that $h(K,K)=r^2$. The operator 
$D^L$ also yields a connection form $\Theta$ in $\Lambda^1M$ such that $ \Theta(K)=1, \ d \Theta=\o_{N}$. 
The two form $\Theta \w rdr$ endows $\D^{+}$ with its natural orientation. But $\D^{+}$ comes also equipped with a complex structure $J_{0}$ satisfying $J_0dr=r\Theta$, which can be extended to $M$ by lifting $I_{N}$ horizontally; we shall still denote it by $J_0$. For suitable functions $F:] 0, \infty[ \to \mathbb{R}$ the closed two-form $\o_{N}+(dJ_0d) F(r)$ gives a K\"ahler structure with respect to $J_0$ that makes $K$ a Hamiltonian Killing vector field. Moreover, by changing the sign of $J_0$ along $\D^{+}$ we obtain an 
almost complex structure $I_0$ that is orthogonal for $g_0$. Parametrising the structure by the moment map $G(r)=1+r F^{\prime}(r)$ of $K$ we may write the metric as
\begin{equation*}
g_0=G(r)g_{N}-r^{-1}G^{\prime}(r)h,
\end{equation*}
with $G$ positive and decreasing in order to define a genuine Riemannian metric. 
The fundamental $2$-forms relative to $J_{0}, I_{0}$ consequently read
\begin{equation*}
\o_{J_0}= G(r)\o_{N}-G^{\prime}(r) \Theta \w dr, \q \o_{I_0}=G(r)\o_{N}+G^{\prime}(r) \Theta \w dr,
\end{equation*}
and the volume form becomes
\begin{equation} \label{vol}
\o_{J_0}^m=-mG^{m-1}(r)G^{\prime}(r) \o_N \w \Theta \w dr.
\end{equation}
The distribution $\D^{+}$ is clearly totally geodesic with respect to  $g_0$ (see also \cite{Hwang-S:}); it is also easy to see that $I_0$ is integrable and that $\D^{+}$ induces a homothetic foliation.

The triple $(M,g_0,J_0)$ thus constructed will be referred to as a \emph{K\"ahler manifold of Calabi type}; we will also allow Riemannian products of K\"ahler manifolds with Riemann surfaces be part of this definition. In real dimension 
$4$  any K\"ahler surface $(M^4,g_0,J_0)$ with a Hamiltonian Killing vector field $K$ for which the almost complex structure $I_0$ reversing the sign of $J_0$ along $\span \{K,J_0K\}$ is locally conformally K\"ahler has to be of Calabi type \cite{WSD}.

Note that the Lee form $\theta_0$ of the complex structure $I_0$, its norm and its dual $\z_0$ 
with respect to $g_0$ satisfy the following relationships:
\begin{equation*}
\begin{split}
&\theta_0= -\frac{2H^{\prime}(r)}{H(r)}dr, \q  \vert \theta_0 \vert^2_{g_0}=4r H^{\prime}(r), \q \z_0= -2r H(r) \frac{d}{dr}
\end{split}
\end{equation*}
where the map $H(r)$ is the reciprocal of $G(r)$and $\frac{d}{dr}$ is the projection of the radial vector field of $\bb{C}\backslash \{0\}$ to $M$.
\subsection{The structure equations} \label{str-eqn}
Let $(M^{2m},g,J), m \geq 2,$ be a K\"ahler manifold endowed with a complex foliation, and denote by $\D^{+}$ the  distribution tangent to its leaves, by $\D^{-}$ the orthogonal complement. We 
assume that $\D^{+}$ induces a homothetic foliation by curves, that is 
\begin{equation} \label{eq-hm}
L_Vg=\theta(V)g
\end{equation}
on $\D^{-}$, whenever $V$ belongs to $\D^{+}$. Moreover $\theta$ in $\Lambda^1 \D^{+}$ satisfies $d\theta=0$ and the real rank of $\D^{+}$ is two. We will denote by $\z$ the vector field dual to $\theta$, that is $\theta=g(\z, \cdot)$.
Decomposing the K\"ahler form $\o_{J}=\o^{+}+\o^{-}$ along $TM=\D^{+} \oplus \D^{-}$ implies 
$$-d\o^{+}=d\o^{-}.$$

The orthogonal almost complex structure 
which reverses the sign of $J$ on $\D^{+}$ will be called $I$; it has K\"ahler form $\o_I=-\o^{+}+\o^{-}$. Recall that a complex distribution $E \subseteq TM$ is holomorphic iff 
$(L_XJ)TM \subseteq E$ for all $X$ in $E$, in which  case the complex Frobenius theorem ensures that $E$ is locally spanned by holomorphic vector fields. 

We begin by showing that the following alternative holds. 
\begin{pro} \label{gl}
A complex homothetic foliation by curves on a connected K\"ahler manifold is either
\begin{itemize}
\item[(i)] holomorphic 
\item[or]
\item[(ii)] totally geodesic and Riemannian.
\end{itemize}
\end{pro}
\begin{proof}
Let $D$ be the orthogonal projection of the Levi-Civita connection onto the splitting $TM=\D^{+} \oplus \D^{-}$. This is a metric and also Hermitian connection, since $(g,J)$ is K\"ahler and the tensor 
$\xi=D-\nabla$ in $\Lambda^1M \otimes \Lambda^{1,1}M$ describes the usual O'Neill obstruction tensors.   Therefore the components of $\xi$ in $\Lambda^1 \D^{+} \otimes \Lambda^{1,1}\D^{+}$ and 
$\Lambda^1 \D^{-} \otimes \Lambda^{1,1}\D^{-}$ vanish and the restriction of $\xi$ to $\D^{+}$ is symmetric since the latter is integrable; moreover from \eqref{eq-hm} and $[\xi,J]=0$ it follows that 
\begin{equation} \label{i-t}
\xi_XY=\mathring{\xi}_XY+\frac{1}{2}\langle JX,Y \rangle J\z+\frac{1}{2}\langle X,Y \rangle \z
\end{equation}
where $\mathring{\xi}_XY+\mathring{\xi}_YX=0$ and $\mathring{\xi}_{X}JY=J\mathring{\xi}_XY$ for all $X,Y$ in $\D^{-}$. Then an easy computation gives $g((L_VJ)X,Y)=-2g(JV, \mathring{\xi}_XY)$ for all $V$ in $\D^{+}$ and whenever 
$X,Y$ are in $\D^{-}$; hence $\D^{+}$ if and only if $\mathring{\xi}$ vanishes. From the closure of $\theta$ 
\begin{equation} \label{closed}
0=d\theta(X,Y)=-\theta[X,Y]=2\theta(\mathring{\xi}_XY)
\end{equation}
for all $X,Y$ in $\D^{-}$.

Now on the open set where $\theta \neq 0$ we have $\D^{+}=span \{\z,J\z\}$ hence the above shows that $\mathring{\xi}=0$. Let us now work on some open set where $\theta$ vanishes. Then the foliation is Riemannian and it has been proved in \cite{nagy2} that 
$\langle \mathring{\xi}_XY, \xi_{V}Z \rangle=0$ for all $X,Y,Z$ in $\D^{-}$ and $V$ in $\D^{+}$. Where $\mathring{\xi} \neq 0$ we have by a dimension argument $span \{\mathring{\xi}_XY: X,Y \in \D^{-} \}=\D^{+}$ showing that the 
restriction of $\xi$ to $\D^{+}$ vanishes, thus the latter is totally 
geodesic; but in this case $(g,I)$ is a nearly K\"ahler structure with intrinsic torsion $\mathring{\xi}$ and canonical hermitian connection $D$ hence $D \mathring{\xi}=0$(see \cite{nagy1} for details).   

By a density argument we conclude that $D\mathring{\xi}=0$ over $M$. Therefore 
either $\mathring{\xi}$ vanishes identically when we obtain (i), or it is nowhere vanishing when again by \eqref{closed}, we get $\theta=0$ and further that $\D^{+}$ is totally geodesic by the considerations above. 
\end{proof}

In the situation of (ii) above the structure $(g,I)$ is nearly K\"ahler having a real two dimensional invariant distribution (here $\D^{+}$) invariant under its canonical Hermitian connection $D$. If $(g,I)$ is K\"ahler it is clear that $\D^{+}$ is parallel 
w.r.t. the Levi-Civita connection of $g$, which is therefore locally the Riemannian product of a Riemann surface and a K\"ahler manifold; otherwise,  
by \cite{nagy1} we have that $(g,J)$ is locally the Riemannian product of a K\"ahler manifold and the twistor space of a quaternion K\"ahler manifold with positive scalar curvature.

Therefore, in order to prove theorem \ref{hmt}, there remains to consider only the case when $\D^{+}$ is holomorphic, an assumption which will be made in the remaining of this section.   
For the local classification of holomorphic homothetic foliations by curves we first look at their structure equations. 
\begin{lema} \label{eqn-mnth} The following hold on the open set where $\theta \neq 0$:
\begin{itemize}
\item[(i)] \label{1} we have  
$$L_{\z}J=(\chi_1 \theta+\chi_2 J\theta) \otimes \z+(-\chi_2\theta+\chi_1 J\theta )\otimes J\z$$ 
for some local functions $\chi_1, \chi_2$ on $M$, with 
$ \chi_1=-2 \vert \theta \vert^{-2} L_{J\z} \ln \vert \theta \vert$;
\item[(ii)] the function $\chi=\chi_1+i\chi_2$ satisfies $\partial_{\D^{-}}\chi=0$ where $\partial_{\D^{-}} : \Lambda^0\D^{-} \to \Lambda^{1,0}\D^{-}$ is the transverse Dolbeault operator;
\item[(iii)] we have $d\o^{-}=\theta \w \o^{-}$; 
\item[(iv)] \label{4} the form $e=\theta+iJ\theta$ in $\Lambda^{0,1}\D^{+}$ satisfies 
$$ de=\eta \o^{+}+e \w \gamma-\overline{e} \w \overline{\gamma}-i\vert \theta \vert^2 \o^{-}
$$
where $\eta : M \to i\bb{R}$ is some function and $\gamma=-2 \partial_{\D^{-}} \ln \vert \theta \vert$;
\item[(v)] in case $\chi_1=0$ there exists a non-zero local function $\varphi$ with $d\varphi \w \theta=0$ such that $\varphi J\z$ is a Hamiltonian Killing vector field for $(g,J)$.
\end{itemize}
\end{lema}
\begin{proof}
(i) Because the tangent distribution $\D^{+}$ is holomorphic we can write $L_{\z}J=\alpha \otimes \z+J\alpha \otimes J\z$, where $\alpha$ is a one-form on $M$. Since $d\theta=0$ and $(g,J)$ is K\"ahler 
we get after anti-symmetrisation  that $-\alpha \w J\theta+J \alpha \w \theta=0$, hence $\alpha$ belongs to $\Lambda^1\D^{+}$ and the desired parametrisation follows. Therefore 
$[\z,J\z]=\vert \theta \vert^2 \chi_1 \z-\vert \theta \vert^2 \chi_2 J\z$, and in particular 
$\vert \theta \vert^4 \chi_1=\theta[\z,J\z]$: the claim follows by taking into account that $d\theta=0$.\\
(ii) Since $J$ is integrable we have $\partial(L_{\z}J)=0$, which gives $d \alpha$ in $\Lambda^{1,1}M$. The claim follows then from $d\theta=0$ and the fact that $d(J\theta)$ lives in $\Lambda^{1,1}M$.\\
(iii) Because $d\o^{-}=-d\o^{+}$ and $\o^{+}$ vanishes on $\D^{-}$ it follows that $d\o^{-}$ has vanishing component in $\Lambda^3\D^{-}$.  Now for any $V_1,V_2$ in $\D^{+}$ we have 
\begin{equation*}
\begin{split}
V_1 \lrcorner V_2 \lrcorner d\o^{-}= V_1 \lrcorner L_{V_2}\o^{-}=L_{V_2}(V_1 \lrcorner \o^{-})+[V_1,V_2] \lrcorner \o^{-}=0
\end{split}
\end{equation*}
by sucesively using that $\o^{-}$ vanishes on the integrable distribution $\D^{+}$. There remains to compute, for $X,Y$ in $\D^{-}$ and $V$ in $\D^{+}$ 
\begin{equation*}
\begin{split}
d\o^{-}(V,X,Y)=(\nabla_V \o^{-})(X,Y)-(\nabla_X \o^{-})(V,Y)+(\nabla_Y \o^{-})(V,X).
\end{split}
\end{equation*}
However $(\nabla_V \o^{-})(X,Y)=(\nabla_V \o^{+})(X,Y)=0$ and $(\nabla_X \o^{-})(V,Y)=-\o^{-}(\nabla_XV,Y)$ hence finally $d\o^{-}(V,X,Y)=-\langle JV, [X,Y] \rangle=\theta(V) \o^{-}(X,Y)$ by using 
\eqref{i-t} and the holomorphy of $\D^{+}$. \\
(iv) The closure of $\theta$ forces $d(J\theta)$ to have type $(1,1)$, hence we can write 
$$d(J\theta)=f\o^{+}+\theta \w \psi+J \theta \w J \psi+g\o^{-}$$ where $\psi$ belongs to $\Lambda^1\D^{-}$ and $f,g$ are functions. Over $\D^{-}$ we have 
$\vert \theta \vert^2\psi=L_{\z}(J\theta)=J L_{\z}\theta=J d_{\D^{-}} \vert \theta \vert^2$ by using that $L_{\z}J$ vanishes on $\D^{-}$. To conclude we note that 
$$d(J\theta) \w \o^{+}=J\theta \w d\o^{+}=\theta \w J \theta \w \o^{-}=
-\vert \theta \vert^2 \o^{+}\w \o^{-}$$
by using (iii).\\
(v) If $\chi_1$ vanishes, part (ii) implies $d\chi_2$ is in $\Lambda^1\D^{+}$; thus after differentiation $d\chi_2([X,Y])=0$ for all $X,Y$ in $\D^{-}$, and so
$$(d \chi_2)\z \langle \z, [X,Y] \rangle+
(d \chi_2)J\z \langle J\z, [X,Y] \rangle=0.$$ 
Since by \eqref{i-t} and $d\theta=0$ we have $\langle \z, [X,Y] \rangle=0$ and $ \langle J\z, [X,Y] \rangle=\vert \theta \vert^2 \langle X,JY \rangle$, it follows that $d\chi_2 \w \theta=0$. Writing locally $\theta=dt$ and 
 $\chi_2=V(t)$, the desired function solves  
$\varphi^{\prime}=\varphi V$, since $L_{\z}J=\chi_2(J\theta \otimes \z-\theta \otimes J\z)$.
\end{proof} 
The case when $\D^{+}$ is totally geodesic falls directly under (v) above; for $\theta$ closed implies $d_{\D^{-}} \vert \theta \vert=0$ hence also $L_{J\z} \vert \theta \vert=0$ as in the proof of previously cited.
\subsection{Integrability and twists}
We now consider deformations of the complex structure $J$ along $\D^{+}$ in view of next section's results. 

Denote the open unit disc in the plane $\{z \in \bb{C} : \vert z \vert<1\}$ with $\bb{D}$, and let 
$$w=w_1+iw_2:M \to \bb{D}$$ 
be a smooth map, which we will call a {\it{twist}}. 
Then define almost complex structures $J_w, I_w$ on $M$ by 
$$J_w=(1-S)J(1-S)^{-1},\q I_w=(1-S)I(1-S)^{-1}$$
where $S$ equals $\biggl (\begin{array}{cc} w_1 & w_2\\
w_2 & -w_1 \end{array} \biggr )$ on $\D^{+}$, in the basis of one-forms $\{\theta,J\theta\}$, and vanishes on $\D^{-}$. Then $I_w,J_w$ are orthogonal for the metric 
$$g_w=g((1+S)(1-S)^{-1} \cdot, \cdot),$$ 
with unchanged K\"ahler forms
 $$g_w(J_w \cdot, \cdot)=\o_J\q \mbox{and}\q g_w(I_w \cdot, \cdot)=\o_I.$$
It is well known that any pair of commuting almost complex structures that preserves $\D^{\pm}$, equals $(I_0,J_0)$ on $\D^{-}$ and is compatible with $\o^{+}$ on $\D^{+}$ is built in such a way.

We shall discuss the conditions under which the structure is preserved by twisting, meaning that $(g_w,J_w)$ is supposed to remain K\"ahler; if this is the case $\D^{+}$ is automatically the leaf-tangent distribution of a 
complex homothetic foliation, as one can check using basic vector fields. 

We will frequently use the M\"obius transformation $w=-\frac{\tilde{w}}{1+\tilde{w}}$ in order to linearise subsequent quadratic equations; under this transformation  
$\bb{D}$ is mapped onto the half-plane $\{z \in \bb{C} : \Re z > -\frac{1}{2}\}$. We denote by $\z_w$ the vector field dual to $\theta$ with respect to  the metric $g_w$ and note that 
$J_w\z_w=J\z$, since the K\"ahler form of $(g_w,J_w)$ does not change. Furthermore, 
\begin{equation} \label{norm-w}
\vert \theta\vert^2_{g_w}=-\frac{\vert 1+w \vert^2}{1-\vert w \vert^2} \vert \theta \vert^2_g=\frac{1}{2\Re \tilde{w}+1} \vert \theta \vert^2_g.
\end{equation}
\begin{lema} \label{def-tw} 
Let $(M,g,J)$ be a K\"ahler manifold equipped with a holomorphic homothetic foliation by curves. 
Given a twist $w$,
\begin{itemize}
\item[(i)] the structure $J_w$ is integrable iff $\partial_{\D^{-}}\tilde{w}+\tilde{w} \gamma=0$;
\item[(ii)] the structure $I_w$ is integrable iff $\overline{\partial}_{\D^{-}}\tilde{w}+(1+\tilde{w}) \overline{\gamma}=0$.
\end{itemize}
\end{lema}
\begin{proof}
(i) As $J$ is integrable and $\D^{+}$ holomorphic, $J_w$ is integrable iff, for any $\alpha$ in $\Lambda^{0,1}_{J_w}\D^{+}$, the form $d\alpha$ has no component in $\Lambda^{2,0}_{J_w}M$;
equivalently $d \alpha \wedge \Lambda^{0,m}_{J_w}M=0$. 
Since the latter space is generated by $\alpha \w \beta$ 
with $\beta$ in $\Lambda^{0,m-1}\D^{-}$ 
we must have $d\alpha \w \alpha \w \beta=0$. Taking now $\alpha=e-w\overline{e}$ yields $\partial_{\D^{-}}w+(w+w^2)\gamma=0$, after a short computation mainly based on (iv) in 
lemma \ref{eqn-mnth}. The claim follows from the explicit 
expression of the transformation $w \mapsto \tilde{w}$ .

Part (ii) is proved similarly, one only has to notice that $\Lambda^{0,1}_{I_w}\D^{+}$ is spanned by $\overline{\alpha}$ with $\alpha$ as above. The integrability of $I_w$ is equivalent to 
$\overline{\partial}_{\D^{-}}w-(1+w) \overline{\gamma}=0$, which yields the claim by M\"obius transformation.
\end{proof}
The above integrability conditions will be used to twist the structure so that $J\z$ is proportional to a Killing vector field first, and then to prove the classification.
\begin{pro} \label{tw1-K}
Let $(M,g,J)$ be a K\"ahler manifold with a holomorphic homothetic foliation. Around points where $\theta \neq 0$ 
there exists a structure-preserving local twist $w$ such that $\varphi(t)J\z$ is a Killing vector field for $g_w$, where $\varphi(t)$ is a one-variable map and $\theta=dt$.
\end{pro}
\begin{proof}
First of all we need to exhibit a twist $w$ rendering $J_w$ integrable, 
for which the previous lemma 
is needed. By (v) in lemma \ref{eqn-mnth} furthermore, in order to have $J\z$ proportional to a Killing field, $L_{J\z} \vert \theta \vert_{g_w}=0$ must hold, that is $\vert \theta \vert^2 \chi_1 (\Re \tilde{w}+\frac{1}{2})+L_{J\z} \Re \tilde{w}=0$ according to \eqref{norm-w}. It is straightforward to verify that 
if $\tilde{\chi}$ satisfies $L_{J\z}\tilde{\chi}=-\frac{1}{2}
\chi$ and $\partial_{\D^{-}}\tilde{\chi}=0$ then $\tilde{w}=\vert \theta \vert^2 \tilde{\chi}$ is a solution to both equations. To finish the proof there remains to prove the existence of such a $\tilde{\chi}$. This comes essentially from the observation that 
$L_{J\z}$ commutes with $d_{\D^{-}}$, hence with its complex counterparts. In fact, it is enough to work along integral sub-manifolds of the integrable distribution $\Ker \theta$, which are Sasakian manifolds with contact form 
$\vert \theta \vert^{-2}J\theta$ and Reeb vector field $J\z$ by (iv) in lemma \ref{eqn-mnth}; in this situation the assertion follows by writing $d_{\D^{-}}$ in local co-ordinates for the Sasakian structure,  and then 
integrating $\chi$ along $J\z$. That we also have $\partial_{\D^{-}} \tilde{\chi}=0$ is due to (ii) in lemma \ref{eqn-mnth}, which asserts that $\partial_{\D^{-}} \chi=0$.
\end{proof}
We may always locally assume that the moment map $z$ of the Killing vector field $X=\varphi(t)J\z$, given by $X \lrcorner\, \o_J=dz$, satisfies $\theta=d\ln z$. That is because for positive functions $a,b$ of $t$ constrained by 
$b^{\prime}+b=-a$, the metric 
\begin{equation} \label{cal-tw}
\hat{g}=ag_{\vert \D^{+}}+bg_{\vert \D^{-}}
\end{equation}
is K\"ahler for $J$, has Lee form $\hat{\theta}=\frac{a}{b}\theta$ and converts $\D^{+}$ in a complex homothetic foliation. The desired relationship between $z$ and the dilation of $\hat{\theta}$ is obtained for 
$b^{\prime}+b(\frac{\varphi^{\prime}+\varphi}{2\varphi})=0$.\\

We are ready now for the main result of this section.
\begin{teo} 
Any K\"ahler manifold with a holomorphic homothetic foliation by curves is locally obtained by twisting a K\"ahler structure of Calabi type. 
\end{teo}
\begin{proof}
We first work on the open set where $\theta \neq 0$ and search for a twist $w$ such that $I_w$ and $J_w$ are both integrable. By lemma \ref{def-tw} this 
amounts to solving $-d_{\D^{-}}\tilde{w}=(1+\tilde{w})\overline{\gamma}+\tilde{w}\gamma$. Then 
$$d_{\D^{-}}(\Re \tilde{w} \vert \theta \vert^{-2})=\vert \theta \vert^{-2} d_{\D^{-}} \ln \vert \theta \vert \q \mbox{and} \q d_{\D^{-}}(\Im \tilde{w} \vert \theta \vert^{-2})=\vert \theta \vert^{-2} Jd_{\D^{-}} \ln \vert \theta \vert,
$$
so that we can look for solutions satisfying $\Re \tilde{w}=\frac{1}{2}(\vert \theta \vert^2-1)$ and 
\begin{equation} \label{final-t}
J d_{\D^{-}} \vert\theta \vert^{-2}=-2d_{\D^{-}}(\Im \ \tilde{w} \vert \theta \vert^{-2}).
\end{equation}
To solve the latter equation, by the argument above and proposition \ref{tw1-K} we are entitled to assume there is a Killing vector field $X$ with $X \lrcorner\, \o_J=dz$ and $\theta=d \ln z$. 

Using the description \cite{LeBrun:nCP2,P-P} of K\"ahler manifolds with a holomorphic Killing vector field $X$,  we write $N$ for the local holomorphic quotient $M \slash \{X,JX\}$, so that to have locally 
$M=\mathbb{R}^2 \times N$ with $X=\frac{d}{ds}, \ X \lrcorner\, \o_J=dz$, where $s,z$ are co-ordinates on $\bb{R}^2$. Moreover,
$$ JX=W^{-1} \frac{d}{dz}, \q Jdz=W^{-1}(ds+\alpha)
$$
where $W^{-1}=g(X,X)$ and $\alpha$, a one-form on $N$ possibly depending on $z$, have zero derivatives $W_s=0$ and $\alpha_s=0$. Then $\o_J=(ds+\alpha) \w dz+\Omega$, with $\Omega$ in $\Lambda^{1,1}N$ such that 
$\Omega(\cdot, J \cdot) >0$. Consequently,
\begin{equation} \label{red-1}
\begin{split}
& d_N\O=0, \q \O_z+d_N \alpha=0,\\
& \alpha_z=-Jd_NW;
\end{split}
\end{equation}
the first line amounts to having $\o_J$ closed, the second describes the integrability of $J$, and $d_N$ is just  the exterior differential on $N$. By differentiating once more we obtain
\begin{equation} \label{red-2}
\O_{zz}-(d_NJd_N)W=0.
\end{equation}
But since $\o^{-}=\O$, imposing $d \o^{-}=d \ln z \w \o^{-}$ yields 
\begin{equation} \label{red-3}
\O=zb
\end{equation}
where $b$ in $\Lambda^{1,1}N$ is closed with $b_s=b_z=0$. In particular $(d_NJd_N)W=0$ by \eqref{red-2}, so $Jd_NW=d_N \tilde{W}$ for some function $\tilde{W}$ on $N$ (possibly depending on $z$). 
To solve \eqref{final-t} we observe that for any function  
$f:M \to \bb{R}$ with $f_s=0$ we have $d_{\D^{-}}f=d_{N}f$; by taking into account that $\vert \theta \vert^{-2}=Wz^{-2}$ we can take $\Im \tilde{w}=-\frac{\tilde{W}}{2W}$.

Now recall that complex homothetic foliations are stable under bi-axial symmetries of type \eqref{cal-tw}, and observe that double twisting is again a twist, because the 
volume form of $\D^{+}$ stays the same; these facts imply that it is always possible to twist the K\"ahler metric so that both $I$ and $J$ become integrable. 

There remains to show that this situation corresponds, locally, to a K\"ahler metric of Calabi type.
By lemma \ref{def-tw} the sign-reversing almost complex structure $I$ is integrable if and only if $\gamma=0$, that is $d_{\D^{-}} \ln \vert \theta \vert=0$ by lemma \ref{eqn-mnth}, (iv). Then $d \ln \vert \theta \vert$ belongs to 
$\Lambda^1 \D^{+}$, and repeating the argument of lemma \ref{eqn-mnth}, (v) yields $d \ln \vert \theta \vert \w \theta=0$. Therefore $\chi_1=0$, and after a transformation of type \eqref{cal-tw} we obtain a Killing vector field such 
that $X \lrcorner\, \omega_J=dz=z \theta$. In the K\"ahler reduction governed by equations \eqref{red-1} this corresponds to $W=q(z)$, so in addition to \eqref{red-3} we also get, from \eqref{red-1}, that $\alpha_z=0$ and 
$-d\alpha=b$. In conclusion $g=q(z)dz^2+\frac{1}{q(z)}(ds+\alpha)^2+zb( \cdot, I \cdot)$ which clearly is the local form of a metric of Calabi type.  

Now on any open set where $\theta=0$, the foliation induced by $\D^{+}$ is Riemannian, and since $(g,J)$ is K\"ahler the considerations at the beginning of section \ref{str-eqn} show that $\D^{-}$ is totally geodesic, in particular integrable.
Writing down the restriction of the K\"ahler structure to $\D^{+}$ using a basis of basic vector fields (with respect to the foliation induced by $\D^{-}$), it is easy to prove 
directly that the structure is locally obtained from the Riemannian product of a Riemann surface and a K\"ahler manifold by means of a twisting endomorphism $S$. Notice 
that when the manifold is complete it is known in this case that both foliations lift to a product foliation on the universal cover \cite{Blum}.
\end{proof}
Conversely, we have showed that K\"ahler manifolds that admit a holomorphic homothetic foliation by curves are obtained locally as follows. 

Let $(M,g_0,J_0)$ be non-product of Calabi type. For any map $w=w_1+iw_2:M \to \bb{D}$ we define the almost complex structures 
$$J_w=(1-S)^{-1}J_0(1-S), I_w=(1-S)^{-1}I_0(1-S),
$$
where $S$ equals $\biggl (\begin{array}{cc} w_1 & w_2\\
w_2 & -w_1 \end{array} \biggr )$ on $\D^{+}$ with basis $\{r^{-1}dr,\Theta \}$, and is the identity on $\D^{-}$, in the notation of section \ref{c-s}. Since $\D^{+}$ is totally geodesic with respect to  $g_0$,  by 
lemma \ref{def-tw} the almost complex structure $J_w$ is integrable if and only if    
\begin{equation} \label{int-w}
d_{\D^{-}}w_2=J_0d_{\D^{-}}w_1,
\end{equation}
which in turn says $w$ is {\it{transversely holomorphic}}. In a basis-free way, this is equivalently re-written as
\begin{equation} \label{in-eq}
(L_{J_0X}S)V=J_0(L_XS)V
\end{equation}
for any $(X,V)$ belonging to $\D^{-} \times \D^{+}$. Similarly, the almost complex structure $I_{w}$ is integrable iff $d_{\D^{-}}w_2=-J_0d_{\D^{-}}w_1$, so in the end $I_w$ and $J_w$ are both complex if and only if $dw \w \theta_0=0$, as expected. 

The Riemannian metric $g_w=g_0((1+S)^{-1}(1-S) \cdot, \cdot)$ is compatible with 
both $J_w$ and $I_w$, and leaves the K\"ahler forms of $(g_0,J_0)$ and $(g_0,I_0)$ unchanged. Therefore, under equation \eqref{int-w} the structure $(g_w,J_w)$ is K\"ahler, and $\D^{+}$ defines a holomorphic homothetic foliation with respect to it.
In the product case, the construction above still applies when the twisting endomorphism $S$, no longer defined in a preferred basis, satisfies \eqref{in-eq}.
 
To sum up, we have actually proved:
\begin{pro} 
Let $(M,g_0,J_0)$ be K\"ahler of Calabi type, and $w$ a transversely holomorphic twist. Then $(g_w,J_w)$ is a K\"ahler structure for which $\D^{+}$ is a complex homothetic foliation. 
\end{pro}

Theorem \ref{hmt} in the introduction is now fully proved. We end this section by computing the Ricci form of some special twisted Calabi-type metrics.
\begin{lema} \label{var-ric} 
Let $(M^{2m},g_0,J_0)$ be of Calabi type with moment map $G(r)=A\,\ln r$ for some constant $A<0$. For any transversely holomorphic twist $w$ such that $L_{\frac{d}{dr}}w=L_{K}w=0$, the Ricci form of the K\"ahler structure $(g_w,J_w)$ is 
$$ \rho^{g_w}=\rho^{g_N}-\frac{1}{2}(dJd) \ln  (1-\vert w \vert^2).
$$
\end{lema}
\begin{proof}
Pick $\varphi$ in $\Lambda^{0,m-1}\D^{-}$ such that $d\varphi=0$ and consider 
$$\Psi=(e_0-w\overline{e_0}) \wedge \varphi$$ 
in $\Lambda^{0,m}_{J_w}M$, where $e_0=r^{-1}dr+i\Theta$ is in $\Lambda^{0,1}_{J_0}\D^{+}$.  Holomorphic transversality, together with the fact that $de_0=-d\overline{e}_0$ belongs to $\Lambda^{1,1}\D^{-}$,  forces  
$ de_0\w \varphi=0=d\overline{e}_ 0 \w \varphi$ and $dw \w \varphi=0$, hence $d\Psi=0$. Moreover 
$\Psi \w \overline{\Psi}=(1-\vert w \vert^2) \Psi_0 \w \overline{\Psi}_0$ where $\Psi_0=e_0 \w \varphi$ belongs to $\Lambda^{0,m}_{J_0}M$; it follows that 
$\vert \Psi \vert^2_{g_w}=(1-\vert w \vert^2) \vert \Psi_0 \vert_{g_0}^2$ since the metrics $g_w$ and $g_0$ are compatible with the same symplectic form. Therefore
$$  \rho^{g_w}=-(dJ_wd) \ln \vert \Psi \vert_{g_w}=-\frac{1}{2}(dJd) \ln  (1-\vert w \vert^2)-(dJ_wd) \ln \vert \Psi_0 \vert_{g_0}.
$$
Because of the choice of moment map, the norm $\vert \Psi_0 \vert_{g_0}$ is given by $\vert \varphi \vert_{g_N}$ up to a constant (see also \eqref{vol}), and the claim follows. 
\end{proof}
Note that these metrics are defined for $0 <r <1$.
\section{Einstein almost K\"ahler metrics} 
\noindent 
Let $\bb{R}^2$ be equipped with co-ordinates $x_1,x_2$,  standard flat metric $g_0$ and standard complex structure $J_0$. Consider $(Z,h,I)$ a K\"ahler manifold and endow $M=\mathbb{R}^2 \times Z$ with the metric and 
orthogonal almost complex structure 
\begin{equation*}
\begin{split} 
g=&(g_0+h)_w, \ J=(J_0+I)_w
\end{split}
\end{equation*}
where $w: Z \to \bb{D}$ is a smooth map. 
Flipping the sign of $J$ along $\span\{\frac{d}{dx_1}, \frac{d}{dx_2}\}$ produces a $g$-orthogonal almost complex structure $\tilde{J}=(-J_0+I)_w$, which is automatically almost K\"ahler, since by construction 
$$g(\tilde{J} \cdot, \cdot)=-dx^1 \w dx^2+\omega^h. $$
By lemma \ref{def-tw} we have that $(g,J)$ is K\"ahler if and only if $w:Z \to \bb{D}$ is holomorphic; in this case the same lemma ensures that $\tilde{J}$ is never integrable for non-constant $w$. These 
assumptions will be valid from now on. 

We begin by showing that the almost K\"ahler structure $(g,\tilde{J})$ belongs to the so-called class $\mathcal{AK}_3$ (see \cite{Gray, Tri2}). This means that the Riemann tensor $R$ preserves the decomposition 
$\Lambda^2M=\lambda^{1,1}_{\tilde{J}}M \oplus \lambda^2_{\tilde{J}}M$, where $\lambda^{1,1}_{\tilde{J}}M$ and $\lambda^{2}_{\tilde{J}}M$ respectively denote the bundles of $\tilde J$-invariant and $\tilde{J}$-anti-invariant two-forms; more explicitly, 
\begin{equation} \label{g3}
R(\tilde{J}X, \tilde{J}Y, \tilde{J}Z, \tilde{J}U)=R(X,Y,Z,U)
\end{equation}
for all $X,Y,Z,U$ in $TM$. 
\begin{teo} \label{ak3} 
The non-K\"ahler, almost K\"ahler manifold $(M,g,\tilde{J})$ belongs to the class $\mathcal{AK}_3$.
\end{teo}
\begin{proof}
As $(g,J)$ is K\"ahler, $R(JX,JY,JZ,JU)=R(X,Y,Z,U)$ for all $X,Y,Z,U$ in $TM$, hence \eqref{g3} is easily seen to be equivalent with 
$$ R(V_1,V_2,V_3,X)=0
$$
and 
$$ R(X,Y,Z,V_4)=0
$$
for arbitrary $V_i, 1 \leq i \leq 4$ in $\D^{+}=span\{\frac{d}{dx_1}, \frac{d}{dx_2}\}$ and $X,Y,Z$ in $\D^{-}=({\D^{+}})^{\perp}$. Since $\D^{+}$ induces a Riemannian foliation with totally geodesic 
orthogonal complement, the second claim follows directly from O'Neill's 
formulas \cite{Besse}. There remains to prove the first assertion.

Let $\nabla$ be the Levi-Civita connection of $g$ and let $\eta=\frac{1}{2}(\nabla \tilde{J})\tilde{J}$ be the intrinsic torsion tensor of $(g,\tilde{J})$. The fact that $(g,J)$ is K\"ahler implies $[\eta,J]=0$, hence the canonical Hermitian connection 
$\bnabla=\nabla+\eta$ preserves $\D^{\pm}$. Therefore, using again O'Neill's formulas we get 
$$ -R(V_1,V_2,V_3,X)=\langle (\bnabla_{V_1}\eta)_{V_2}V_3-(\bnabla_{V_2}\eta)_{V_1}V_3,X\rangle 
$$
for all $V_i, 1 \leq i \leq 3$ in $\D^{+}$ and any $X$ in $\D^{-}$.

Now let $K_i=\frac{d}{dx_i}, 1 \leq i \leq 2$. Since $w$ is defined on $Z$, by construction of $g$ we have that $K_1, K_2$ are Killing vector fields for $g$, and holomorphic for $J$ and $\tilde{J}$, that is $L_{K_i}g=0$ and 
$L_{K_i}J=L_{K_i}\tilde{J}=0, i=1,2$. But $K_1,K_2$ commute, so $2g(\nabla_{K_i}K_j,X)=-Xg (K_i,K_j)$ and  hence  
\begin{equation} \label{prel-t}
2g(\eta_{K_i}K_j,X)=Xg(K_i,K_j)
\end{equation} for $1\leq i,j \leq 2$ and any $X$ in $\D^{-}$,
since $\D^{+}$ is parallel with respect to  $\bnabla$. 

Since the metric $g$ only depends on $Z$, the Koszul formula yields $g(\nabla_{K_i}K_j,K_k)=0$ for all $1 \leq i,j,k \leq 2$; moreover $\eta_U \tilde{J}+\tilde{J}\eta_U=0$ for all $U$ in $TM$, and thus rank\,$\D^{+}=2$ implies 
$g(\eta_{K_i}K_j,K_k)=0, 1 \leq i,j,k \leq 2$. Again by the $\bnabla$-parallelism of $\D^{+}$ it follows that $\bnabla_{K_i}K_j=0, 1 \leq i,j \leq 2$.  Differentiating \eqref{prel-t} gives 
$(\bnabla_{K_i}\eta)_{K_j}K_k=0, 1 \leq i,j,k \leq 2$; in particular, $R(K_i,K_j,K_k,X)$ vanishes when $X$ is in $\D^{-}$. 
\end{proof}
We record now some details on the algebraic structure of the intrinsic torsion $\eta$ of $(g,\tilde{J})$. Because $(g,J)$ is K\"ahler, $\D^{-}$ is totally geodesic 
and $\D^{+}$ has real rank two the K\"ahler nullity 
$$ \{U \in TM: \eta_U=0\}=\D^{-}
$$
on the open set where $\eta$(or equivalently $dw$ by \eqref{prel-t}) does not vanish. Easy orthogonality and dimension arguments also show that 
$$ \eta_{\D^{+}}\D^{+} \subseteq \D^{-}, \ \eta_{\D^{+}} \D^{-}=\D^{+}
$$
on the open set where $dw \neq 0$. 

Generically the metrics in theorem \ref{ak3} are locally irreducible as the following shows. 
\begin{pro} \label{l-i}
The Riemannian manifold $(M,g)$ is locally irreducible, unless locally $Z=Z_1 \times Z_2$, a Riemannian product of K\"ahler manifolds such that $dw_{\vert TZ_2}=0$. 
\end{pro}
\begin{proof}Assume that $g$ is locally reducible and let $E\subseteq TM$ be parallel w.r.t. the Levi-Civita connection of $g$. Because $(g,J)$ is K\"ahler by a standard holonomy argument we may 
assume that $E$ is stable under $J$, unless $g$ is flat; since it would imply the integrability of $\tilde{J}$(see \cite{Go}) this case can 
be excluded. We obtain a symmetric endomorphism $S$ of $M$ such that $\nabla S=0$ and $[S,J]=0$ given by the identity on $E$ and its opposite on $E^{\perp}$. 
After splitting $S=S^{\prime}+S^{\prime \prime}$ into $\tilde{J}$-invariant resp. $\tilde{J}$-
anti-invariant components it follows that 
\begin{equation} \label{par-S}
\bnabla S^{\prime}=[\eta, S^{\prime \prime}], \ \bnabla S^{\prime \prime}=[\eta, S^{\prime}].
\end{equation}
Since the Riemann curvature acts trivially on $S$ it is straightforward to get from \eqref{g3} that 
$$ [R(U_1,U_2),S^{\prime}]=[R(U_1,U_2),S^{\prime \prime}]=0
$$
for all $U_1,U_2$ in $TM$. Now using the O'Neill's formulas for the remaining 
curvature terms, we obtain, as in the proof of 
theorem \ref{ak3} that 
$R(V,W)=-[\eta_V, \eta_W]$
for all $V,W$ in $\D^{+}$. Because $(g,\tilde{J})$ is almost-K\"ahler we have $\eta_{\tilde{J}U}=\eta_U \tilde{J}$ for all $U$ in $TM$, in particular $\eta_V^2 S^{\prime \prime}+S^{\prime \prime} \eta^2_V=0$ for all $V$ in $\D^{+}$; since 
$S^{\prime \prime} \D^{\pm} \subseteq \D^{\mp}$ and $\eta^2_V \leq 0$ preserves $\D^{\pm}$ it follows that $S^{\prime \prime}$ vanishes on the open set where $\eta \neq 0$ and therefore everywhere by \eqref{par-S}. Therefore 
$S$ is invariant under both $J$ and $\tilde{J}$ and because it is symmetric with $S^2=1_{TM}$ we can assume w.l.o.g. that $S_{\vert \D^{+}}=1_{\D^{+}}$. From $\bnabla S=0$ and $[\eta,S]=0$ we get that the distribution $\ker(S+1) \subseteq 
\D^{-}$ satisfies $\eta(\ker(S+1))=0$ and it is preserved by the connection $\bnabla$. An elementary projection argument and \eqref{prel-t} lead now to a local Riemannian splitting of $Z$ as in the statement.   
\end{proof}
The above $\mathcal{AK}_3$-structures made their first appearance, when $Z$ is a Riemann surface, in \cite{AAD}; the same authors \cite{Apostolov-AD:integrability}  showed that they 
exhaust the class of $4$-dimensional $\mathcal{AK}_3$-manifolds. 
Recall that in dimension $4$ these metrics are Einstein (Ricci-flat) precisely when they arise, in the form of the Gibbons-Hawking Ansatz, from a translation-invariant harmonic 
function; see \cite{Apostolov-AD:integrability, Nurowski-P:aK-nK}.\\

We now construct locally irreducible examples, in arbitrary dimensions, where the almost K\"ahler metrics of theorem \ref{ak3} are Einstein (Ricci-flat). Relying on a similar computation to that of lemma \ref{var-ric}, it is equivalent to 
provide K\"ahler manifolds $(Z,h,I)$ satisfying
\begin{equation} \label{ric}
\begin{split}
\rho^h=& \frac{1}{2}(dId) \ln (1-\vert w \vert^2)=\frac{1}{2}(dId) \ln \Re \hat{w},
\end{split}
\end{equation} 
where $\hat{w}=\frac{1-w}{1+w}$ for some non-constant holomorphic map 
$w :Z \to \bb{D}$. 
\begin{ex} \label{ex0}
The fact that the curvature condition \eqref{ric} is stable under the Calabi Ansatz, in the following sense, is the key.  

In general, if we take a K\"ahler manifold $N$ and use Calabi's recipe to produce $(Z^{2n},h,I)$ with  moment map $G(r)=A\,\ln r$, as explained in section \ref{c-s}, then $ \rho^{g_0}=\rho^{g_N}$ by lemma \ref{var-ric}. Therefore if 
$(g_N,I_N)$ satisfies \eqref{ric} for some holomorphic map $w:N \to \bb{D}$, so does $(h,I)$ after lifting $w$ to $Z$.   

Now suppose the base K\"ahler manifold is a Riemann surface $(\Sigma,g_{\Sigma}, I_{\Sigma})$ equipped with a non-constant holomorphic function $w:\Sigma \to \bb{D}$; equation \eqref{ric} is satisfied 
when $(1-\vert w \vert^2)^{-1}g_{\Sigma}$ is flat. 
Assuming this, we iterate the general construction $n-1$ times while retaining the same type of moment map. The resulting  $2n$-dimensional K\"ahler manifold 
 satisfies \eqref{ric}, and possesses a local $T^{n-1}$ action by holomorphic isometries. 
\end{ex} 

Another family of examples arises from observing that \eqref{ric} interacts equally well with the twisted Calabi construction, making again iteration possible. 
\begin{ex} \label{ex1}
Let $(Z_{k},g_k,I_k)$ be obtained by the twisted Calabi construction with base $(Z_{k-1},g_{k-1},I_{k-1})$ as in lemma \ref{var-ric} and holomorphic twist $w:Z_{k-1} \to \bb{D}$. Therefore if $\rho^{g_{k-1}}=\frac{m-k+1}{2}(dI_{k-1}d) \ln (1-\vert w \vert^2) $
we have $\rho^{g_k}=\frac{m-k}{2}(dI_{k}d) \ln (1-\vert w \vert^2).$
The notation is obviously meant to hint at an inductive process. Let the first step ($k=1$) be a Riemann surface $(\Sigma,g_{\Sigma},I_{\Sigma})$ with a non-constant holomorphic function $w:\Sigma \to \bb{D}$  such that 
$\rho^{g_{\Sigma}}=\frac{m}{2}(dI_{\Sigma}d) \ln (1-\vert w \vert^2)$; this amounts to $(1-\vert w \vert^2)^{-m}g_{\Sigma}$ being a flat metric. 
The $(m-1)$-fold iteration of the procedure eventually generates a K\"ahler manifold  
$(Z_{m-1},g_{m-1},I_{m-1})$ of complex dimension $m$ satisfying equation \eqref{ric}. Once again there is a $(m-1)$-torus local action by holomorphic isometries.  Also notice that the next step $(Z_m,g_m,I_m)$ of the iteration is Ricci-flat.   

\end{ex}
The previous two groups of examples are distinguished by the behaviour of the distribution $\Ker dw$. This is of codimension-two, integrable and holomorphic (clearly where $dw \neq 0$). For 
any $(Z^{2m},h,I)$ subject to \eqref{ric}, $\Ker dw$ is tangent to the leaves of a homothetic foliation in both cases \ref{ex0} and \ref{ex1}; but whereas in the former  it is totally geodesic, this not the case in the latter.

\vspace{1cm}

\vspace{1cm} 

\frenchspacing

\begin{thebibliography}{99}
\bibitem{Al}
D. Alekseevsky, \textit{Quaternion Riemann spaces with transitive or solvable group of motions}, Functional Anal.Appl. {\bf{4}}(1970), 321-322.
\bibitem{al}
B. Alexandrov, G. Grantcharov, S. Ivanov, \textit{Curvature properties of twistor spaces of quaternionic K\"ahler manifolds},
J. Geom. {\bf{69}} (1998), 1-12.
%
\bibitem{ADM}
V. Apostolov, T. Dr\u{a}ghici, A. Moroianu, \textit{A splitting theorem for K\"ahler manifolds with constant eigenvalues
of the Ricci tensor}, Int. J. Math. {\bf{12}} (2001), 769-789.
%
\bibitem{AAD}
V. Apostolov, J. Armstrong, T. Dr\u{a}ghici, 
\textit{Local rigidity of certain classes of almost K\"ahler $4$-manifolds}, Ann. Glob. Anal. Geom. {\bf{21}} (2002), 151-176.
%
\bibitem{Apostolov-AD:integrability}
\bysame, 
\textit{Local models and integrability of certain almost K\"ahler 4-manifolds}, {Math. Ann.}  {\bf 323}  (2002),  no. 4, 633--666.
%
\bibitem{WSD}
V. Apostolov, D. M. J. Calderbank, P. Gauduchon, \textit{The geometry of weakly selfdual K\"ahler surfaces}, Compositio Math. {\bf{135}} (2003), 279-322.
%
\bibitem{Baird-Wood-book} 
P. Baird, J. Wood, Harmonic morphisms between Riemannian manifolds, \textit{London Math. Soc. Monographs}, no.29, Oxford Univ. Press, Oxford, 2003.
%
\bibitem{Besse}
A. Besse, \textit{Einstein manifolds}, reprint of the 1987 edition. Classics in Mathematics. Springer-Verlag, Berlin, 2008.
%
\bibitem{Blum} 
R. A. Blumenthal, J. J. Hebda, \textit{De Rham decomposition theorems for foliated manifolds}, Ann. Inst. Fourier (Grenoble) {\bf{33}}, no.2 (1983), 183-198. 
%
%
\bibitem{Bryant}
R. Bryant, \textit{Harmonic morphisms with fibers of co-dimension one}, Comm. Anal. and Geom. {\bf{8}}, no.2 (2000), 219-265.
%
%
\bibitem{cal}
E. Calabi, \textit{Extremal K\"ahler metrics}, in Seminar on Differential Geometry, \textit{Ann. of Math. Stud.}  102, Princeton Univ. Press, Princeton, N.J., 1982, pp. 259-290.
%
\bibitem{Go}
S. I. Goldberg, \textit{Integrability of almost K\"ahler manifolds}, Proc. Amer. Math. Soc. {\bf{21}} (1969), 96-100.
%
\bibitem{Gray}
A. Gray, \textit{Curvature identities for Hermitian and almost Hermitian manifolds}, T\^ohoku Math. J. 
{\bf{28}} (1976), 601-612.
%
\bibitem{Hwang-S:}
A. Hwang, M. Singer, \textit{A momentum construction for circle-invariant K\"ahler metrics}, Trans. Amer. Math. Soc. {\bf{354}} (2002), no. 6, 2285-2325. 
%
\bibitem{LeBrun:nCP2}
C. R. LeBrun, \textit{Explicit selfdual metrics on {$\mathbb{CP}_2\#\dots\#\mathbb{CP}_2$}}, 
{J. Diff. Geom.} {\bf 34} (1991), {223--253}.
%
%
\bibitem{nagy1} 
P.-A.Nagy, \textit{On nearly K\"ahler geometry}, Ann. Global Anal. Geom. 
{\bf{22}} (2002), no. 2, 167-178.
%
\bibitem{nagy2} 
P.-A. Nagy, \textit{Rigidity of Riemannian foliations with complex leaves on K\"ahler manifolds}, J. Geom. Anal. {\bf{13}}(2003), no.4, 659-667.
%
\bibitem{Nurowski-P:aK-nK}
P. Nurowski, M. Przanowski, \textit{A four-dimensional example of a Ricci
flat metric admitting almost K\"ahler non-K\"ahler structure},  Class. Quantum Gravity \textbf{16} (1999), no. 3, L9--L13. 
%
\bibitem{pant}
R. Pantilie, \textit{Harmonic morphisms with one-dimensional fibres}, Internat. J. Math. {\bf{10}}  (1999), no.4, 457-501.
%
\bibitem{P-P}
H. Pedersen, Y. S. Poon, \textit{Hamiltonian constructions of K\"ahler-Einstein metrics and K\"ahler metrics of constant scalar curvature}, Comm. Math. Phys. {\bf{136}} (1991), 309-326.
%
\bibitem{Pedersen-Swann}
H.Pedersen, A.Swann \textit{Riemannian submersions, four-manifolds and Einstein-Weyl geometry}, Proc.London Math.Soc. {\bf{66}}(1993), 381-399.
%
\bibitem{sven}
M. Svensson, \textit{Holomorphic foliations, harmonic morphisms and the Walczak formula}, J. London Math. Soc. (2) {\bf{68}} (2003), 781-794.
%
\bibitem{Tri2}
F. Tricerri, L. Vanhecke, \textit{Curvature tensors on almost Hermitian manifolds}, Trans. Amer. Math. Soc. {\bf{267}} (1981), 365-398.
%
\bibitem{V2}
I. Vaisman, \textit{Conformal foliations}, Kodai J. Math. {\bf{2}} (1979), 26-37.
%
\bibitem{wood}
J. C. Wood, \textit{Harmonic morphisms and Hermitian structures on Einstein
  $4$-manifolds}, International J. Math. {\bf 3} (1992), 415--439.
%
\end{thebibliography}
\end{document}